\documentclass[11pt]{amsart}

\usepackage{amsmath}
\usepackage{amssymb}
\usepackage{fancybox}
\usepackage{eucal}
\usepackage{amsthm}
\usepackage{amscd}
\usepackage{wasysym}
\usepackage{url}
\usepackage{MnSymbol} 
\usepackage{natbib}

\usepackage{epsfig}

\usepackage{amssymb,}
\usepackage[]{amsmath, amsthm, amsfonts,graphicx, amscd,}
\usepackage[all,cmtip]{xy}
\usepackage{setspace, xspace}

\usepackage{enumitem}

\usepackage{hyperref}

\newcommand{\concat}[0]{\textrm{\^{}}}

\newcommand{\comp}[0]{\mathtt{c}}

\DeclareMathSymbol{\mlq}{\mathord}{operators}{``}
\DeclareMathSymbol{\mrq}{\mathord}{operators}{`'}
\DeclareMathOperator{\dom}{dom}
\DeclareMathOperator{\ran}{ran}

\DeclareMathOperator{\Succ}{Succ}
\DeclareMathOperator{\Indep}{Indep}

\newtheorem{thm}{Theorem}[section]
\newtheorem{theorem}[thm]{Theorem}
\newtheorem{corollary}[thm]{Corollary}
\newtheorem{cor}[thm]{Corollary}

\newtheorem{proposition}[thm]{Proposition}

\newtheorem{lemma}[thm]{Lemma}

\newtheorem*{claim}{Claim}
\theoremstyle{remark}

\newtheorem{remark}[thm]{Remark}
\newtheorem{np*}{Non-Proof}

\theoremstyle{definition}
\newtheorem{defn}[thm]{Definition}
\newtheorem{definition}[thm]{Definition}

\newtheorem{notation}[thm]{Notation}
\newtheorem{exam}[thm]{Example}

\numberwithin{subcase}{case}

\newcommand{\mc}{\mathcal }

\newcommand*{\scrA}{\mathcal{A}}
\newcommand*{\scrB}{\mathcal{B}}

\newcommand*{\scrC}{\mathcal{C}}

\newcommand*{\scrN}{\mathcal{N}}

\newcommand{\bc}{\textbf{c}}
\newcommand{\bd}{\textbf{d}}
\newcommand{\be}{\textbf{e}}

\newcommand{\abar}{\bar{a}}
\newcommand{\bbar}{\bar{b}}
\newcommand{\cbar}{\bar{c}}

\newcommand{\rt}{\rightarrow}

\newcommand{\iso}{\cong}
\newcommand{\niso}{\ncong}

\newcommand{\ce}{c.e.\ }
\newcommand{\join}{\oplus}

\newcommand{\Eres}{\upharpoonright\!\!\upharpoonright}
\newcommand{\res}{\upharpoonright}
\newcommand{\la}{\langle}
\newcommand{\ra}{\rangle}
\newcommand{\conc}{\widehat{\ }}

\begin{document}

\setlist[enumerate]{noitemsep, topsep=0pt}

\title{Degrees of Categoricity on a Cone}

\author{Barbara F. Csima}
\address{Department of Pure Mathematics\\
University of Waterloo\\
 Canada}
 \email{csima@uwaterloo.ca}
 \urladdr{\href{http://www.math.uwaterloo.ca/~csima}{www.math.uwaterloo.ca/$\sim$csima}}

\author{Matthew Harrison-Trainor}
\address{Group in Logic and the Methodology of Science\\
University of California, Berkeley\\
 USA}
\email{matthew.h-t@berkeley.edu}
\urladdr{\href{http://www.math.berkeley.edu/~mattht/index.html}{www.math.berkeley.edu/$\sim$mattht}}

\thanks{B. Csima was partially supported by Canadian NSERC Discovery Grant 312501.\\
M. Harrison-Trainor was partially supported by the Berkeley Fellowship and NSERC grant PGSD3-454386-2014.}

\begin{abstract}
We investigate the complexity of isomorphisms of computable structures on cones in the Turing degrees. We show that, on a cone, every structure has a strong degree of categoricity, and that degree of categoricity is $\bf{0^{(\alpha)}}$ for some $\alpha$. To prove this, we extend Montalb\'an's $\eta$-system framework to deal with limit ordinals in a more general way. We also show that, for any fixed computable structure, there is an ordinal $\alpha$ and a cone in the Turing degrees such that the exact complexity of computing an isomorphism between the given structure and another copy $\scrB$ in the cone is a \ce degree in $\Delta^0_\alpha(\scrB)$. In each of our theorems the cone in question is clearly described in the beginning of the proof, so it is easy to see how the theorems can be viewed as general theorems with certain effectiveness conditions.
\end{abstract}

\maketitle


\section{Introduction}

In this paper, we will consider the complexity of computing isomorphisms between computable copies of a structure after relativizing to a cone. By relativizing to a cone, we are able to consider \textit{natural} structures, that is, those structures which one might expect to encounter in normal mathematical practice. The main result of this paper is a complete classification of the natural degrees of categoricity: the degrees of categoricity of natural computable structures. Unless otherwise stated, all notation and conventions will be as in the book by Ash and Knight \cite{AshKnight00}. We consider countable structures over at most countable languages.

Recall that a computable structure is said to be computably categorical if any two computable copies of the structure are computably isomorphic. As an example, consider the rationals as a linear order; the standard back-and-forth argument shows that the rationals are computably categorical. It is easy to see, however, that not all computable structures are computably categorical. The natural numbers as a linear order is one example.

There has been much work in computable structure theory dedicated to characterizing computable categoricity for various classes of structures (e.g., a linear order is computably categorical if and only if it has at most finitely many successivities \cite{GoncharovDzgoev}, \cite{Remmel}). For those structures that are not computably categorical, what can we say about the isomorphisms between computable copies, or more generally, about the complexities of the isomorphisms relative to that of the structure?

We can extend the definition of computable categoricity as follows.

\begin{definition} A computable structure $\scrA$ is
\emph{$\bf{d}$-computably categorical} if for all
computable $\scrB \iso \scrA$ there exists a $\bf{d}$-computable isomorphism between $\scrA$ and $\scrB$.
\end{definition}

\noindent It is easy to see, for example, that the natural numbers as a linear order, $\scrN$, is $\bf{0}'$-computably categorical. Indeed, it is also easy to construct a computable copy $\scrA$ of $\scrN$ such that every isomorphism between $\scrA$ and $\scrN$ computes $\bf{0}'$. Thus $\bf{0}'$ is the least degree $\bf{d}$ such that $\scrN$ is $\bf{d}$-computably categorical. This motivates the following definitions.

\begin{definition}We say a computable structure $\scrA$ has \emph{degree of
categoricity} $\bf{d}$ if

\begin{enumerate}
\item $\scrA$ is $\bf{d}$-computably categorical.
\item If $\scrA$ is $\bf{c}$-computably categorical, then
    $\bf{c} \geq \bf{d}$.
\end{enumerate}

\end{definition}

\begin{definition}We say that a Turing degree $\bf{d}$ is a \emph{degree of
categoricity} if there exists a computable structure $\scrA$ with degree of
categoricity $\bf{d}$.
\end{definition}

The notion of a degree of categoricity was first introduced by Fokina, Kalimullin and R. Miller \cite{FKM}. They showed that if ${\bf d}$ is d.c.e.\  (difference of c.e.) in and above ${\bf{0}}^{(n)}$, then ${\bf d}$ is a degree of categoricity. They also showed that ${\bf 0}^{(\omega)}$ is a degree of categoricity. For the degrees c.e.\ in and above ${\bf0}^{(n)}$, they exhibited rigid structures capturing the degrees of categoricity. In fact, all their examples had the following, stronger property.

\begin{definition}A degree of categoricity ${\bf d}$ is a
\emph{strong} degree of categoricity if there is a structure
$\scrA$ with computable copies $\scrA_0$ and $\scrA_1$ such
that ${\bf d}$ is the degree of categoricity for $\scrA$, and
every isomorphism $f: \scrA_0 \rightarrow \scrA_1$ satisfies
deg$(f) \geq {\bf d}$.
\end{definition}


In \cite{CFS}, Csima, Franklin and Shore showed that for every computable
ordinal $\alpha$, ${\bf 0}^{(\alpha)}$ is a strong degree of
categoricity. They also showed that if $\alpha$ is a computable
successor ordinal and ${\bf d}$ is d.c.e. in and above ${\bf
0}^{(\alpha)}$, then ${\bf d}$ is a strong degree of
categoricity.

In \cite{FKM} it was shown that all strong degrees of categoricity are hyperarithmetical, and in \cite{CFS} it was shown that all degrees of categoricity are hyperarithmetical. There are currently no examples of degrees of categoricity that are not strong degrees of categoricity. Indeed, we do not even have an example of a structure that has a degree of categoricity but not strongly.

All known degrees of categoricity satisfy ${\bf0}^{(\alpha)}
\leq {\bf d} \leq {\bf 0}^{(\alpha +1)}$ for some computable
ordinal $\alpha$.  So in particular, all known
non-computable degrees of categoricity are hyperimmune. In \cite{AC}, Anderson and Csima showed that no non-computable
hyperimmunefree degree is a degree of
categoricity. They also showed that there is a $\Sigma^0_2$ degree that is not
a degree of categoricity, and that if $G$ is 2-generic (relative to a perfect tree), then deg$(G)$
is not a degree of categoricity. The question of whether there exist $\Delta^0_2$ degrees that are not degrees of categoricity remains open.

Turning to look at the question of degree of categoricity for a given structure, R. Miller showed that there exists a field that does not have a degree of categoricity \cite{Miller2009}, and Fokina, Frolov, and Kalimullin \cite{FFK} showed that there exists a rigid structure with no degree of categoricity.

In this paper, we claim that the only \emph{natural} degrees of categoricity are those of the form ${\bf0}^{(\alpha)}$ for some computable ordinal $\alpha$. By a natural degree of categoricity, we mean the degree of categoricity of a natural structure.

What do we mean by \emph{natural}? A natural structure is one which might show up in the normal course of mathematics; a structure which has been constructed, say via diagonalization, to have some computability-theoretic property is \textit{not} a natural structure. Of course, this is not a rigorous definition. Instead, we note that arguments involving natural structures tend to relativize, and so a natural structure will have property $P$ if and only if it has property $P$ on a cone (i.e., there is a Turing degree $\textbf{d}$ such that for all $\textbf{c} \geq \textbf{d}$, $P$ holds relative to $\textbf{c}$). Thus by considering arbitrary structures on a cone, we can prove results about natural structures.

The second author previously considered degree spectra of relations on a cone \cite{HT}. McCoy \cite{McCoy02} has also shown that on a cone, every structure has computable dimension $1$ or $\omega$. Here, we give an analysis of degrees of categoricity along similar lines.

Our main theorem is:
\begin{theorem}\label{thm:main1}
Let $\mc{A}$ be a countable structure. Then, on a cone: $\mc{A}$ has a strong degree of categoricity, and this degree of categoricity is ${\bf0}^{(\alpha)}$. 
\end{theorem}

\noindent There are three important parts to this theorem: first, that every natural structure has a degree of categoricity; second, that this degree of categoricity is a strong degree of categoricity; and third, that the degree of categoricity is of the particular form ${\bf0^{(\alpha)}}$. The ordinal $\alpha$ is the least ordinal $\alpha$ such that $\scrA$ is ${\bf0}^{(\alpha)}$-computably categorical on a cone. This is related to the Scott rank of $\mc{A}$ under an appropriate definition of Scott rank \cite{MontalbanSR}: $\alpha$ is the least ordinal $\alpha$ such that $\mc{A}$ has a $\Sigma_{\alpha+2}^{\inf}$ Scott sentence if $\alpha$ is infinite (or a $\Sigma_{\alpha + 3}$ Scott sentence if $\alpha$ is finite). While $\alpha$ may not be computable, every ordinal is computable on some cone.

The construction of a structure with degree of categoricity some d.c.e.\ (but not c.e.) degree uses a computable approximation to the d.c.e.\ degree; this requires the choice of a particular index for the approximation, and hence the argument that the resulting structure has degree of categoricity d.c.e.\ but not c.e.\ does not relativize. By our theorem, there is no possible construction which does relativize. Moreover, our theorem says something about what kinds of constructions would be required to solve the open problems about degrees of categoricity, for example whether there is a 3-c.e.\ but not d.c.e.\ degree of categoricity, or whether there is a degree of categoricity which is not a strong degree of categoricity---the proof must be by constructing a structure which is not natural, using a construction which does not relativize.

The proof of the Theorem \ref{thm:main1} also gives an effectiveness condition which, if it holds of some computable structure, means that the conclusion of the theorem is true of that structure without relativizing to a cone.

\begin{corollary}
The degrees of categoricity on a cone are those of the form $\bf{0}^{(\alpha)}$.
\end{corollary}

Indeed, each degree of the form $\bf{0}^{(\alpha)}$ is a degree of categoricity on a cone. To see this, examine the proof of Theorem 3.1 of \cite{CFS} showing that each $\bf{0}^{(\alpha)}$ is a degree of categoricity, and note that the proof relativizes.

In 2012, Csima, Kach, Kalimullin and Montalb\'an had an unpublished proof of Theorem \ref{thm:main1} in the case where $\mc{A}$ is $\bf{0'}$-computably categorical on a cone. That is, they showed that if $\scrA$ is $\bf{0'}$-computably categorical on a cone, but not computably categorical on a cone, then $\scrA$ has strong degree of categoricity $\bf{0'}$ on a cone. They also conjectured the general result at that time.

The second result of this paper concerns the difficulty of computing isomorphisms between two given computable copies $\mc{A}$ and $\mc{B}$ of a structure. We show that, on a cone, there is an isomorphism of least degree between $\mc{A}$ and $\mc{B}$, and that it is of c.e.\ degree.
\begin{theorem}\label{thm:main2}
Let $\mc{A}$ be a countable structure. On a cone: if $\scrA$ is  $\Delta^0_\alpha$ categorical, then for every copy $\scrB$ of $\scrA$, there is a  degree \textup{\textbf{d}} that is $\Sigma^0_{\alpha-1}$ in $\scrB$ if $\alpha$ is a successor, or $\Delta^0_\alpha$ in $\scrB$ if $\alpha$ is a limit, such that \textup{\textbf{d}} computes an isomorphism between $\scrA$ and $\scrB$ and such that all isomorphisms between $\scrA$ and $\scrB$ compute \textup{\textbf{d}}.
\end{theorem}
\noindent The degree $\textbf{d}$ is the least degree of an isomorphism between $\mc{A}$ and $\mc{B}$.

We begin in Section \ref{sec:cone} by giving the technical definitions for what we mean by ``on a cone.'' In Section \ref{sec:ce-degree} we prove \ref{thm:main2}. In Section \ref{sec:notcconanycone} we prove a stronger version of Theorem \ref{thm:main1} in the restricted case of structures which are $\bf{0'}$-computably categorical on a cone; it will follow that the only possible degrees of categoricity on a cone for such structures are $\bf{0}$ and $\bf{0'}$. In order to prove the general case of Theorem \ref{thm:main1}, we need to use the method of $\alpha$-systems. These were introduced by Ash, see \cite{AshKnight00}. Montalb\'an \cite{Montalban} introduced $\eta$-systems, which are similar to Ash's $\alpha$-systems but give more control. They also deal with limit ordinals in a different way. We need the extra control of Montalb\'an's $\eta$-systems, but we need to deal with limit ordinals as in Ash's $\alpha$-systems. So in Section \ref{sec:metatheorem} we introduce a modified version of Montalb\'an's $\eta$-systems. We conclude in Section \ref{sec:mainthm} with a complete proof of Theorem \ref{thm:main1}.

\section{Relativizing to a Cone}\label{sec:cone}

A \textit{cone of Turing degrees} is a set $C_{\textbf{d}} = \{ \textbf{c} : \textbf{c} \geq \textbf{d} \}$. Martin \cite{Martin68} showed that under set-theoretic assumptions of determinacy, every set of Turing degrees either contains a cone or is disjoint from a cone. Noting that every countable intersection of cones contains a cone, we see that we can form a $\{0,1\}$-valued measure on sets of degrees by assigning measure one to those sets which contain a cone. In this paper, all of the sets of degrees which we will consider arise from Borel sets, and  by Borel determinacy \cite{Martin75}, such sets either contain or are disjoint from a cone.

If $P$ is a statement which relativizes to any degree, we say that $P$ holds on a cone if there is a degree $\textbf{d}$ (the base of the cone) such that for all $\textbf{c} \geq \textbf{d}$, $P$ holds relative to $\textbf{c}$. Thus a statement holds on a cone if and only if it holds almost everywhere relative to the Martin measure. In the rest of this section, we will relativize the definitions we are interested in.

\begin{definition}\label{def:cconacone} The structure $\scrA$ is \emph{computably categorical on the cone above} $\textbf{d}$ if for all $\textbf{c} \geq \textbf{d}$ whenever $\scrB$ and $\scrC$ are $\textbf{c}$-computable copies of $\scrA$, there exists a $\textbf{c}$-computable isomorphism between $\scrB$ and $\scrC$. More generally, a structure is \emph{$\Delta^0_\alpha$ categorical on the cone above} $\textbf{d}$ if for all $\textbf{c} \geq \textbf{d}$ whenever $\scrB$ and $\scrC$ are $\textbf{c}$-computable copies of $\scrA$, there exists a $\Delta^0_\alpha(\textbf{c})$-computable isomorphism between $\scrB$ and $\scrC$.
\end{definition}

Note that even if $\alpha$ is not computable, there is a cone on which $\alpha$ is computable, and for $\textbf{c}$ on this cone, $\Delta^0_\alpha(\textbf{c})$ makes sense. In a similar way, we do not have to assume that the structure $\mc{A}$ is computable.

Recall that a computable structure $\scrA$ is relatively $\Delta^0_\alpha$ categorical if for all $\scrB \iso \scrA$, some isomorphism from $\scrA$ onto $\scrB$ is $\Delta^0_\alpha(\scrB)$, and that there exist structures that are $\Delta^0_\alpha$ categorical but not relatively so \cite{GHKMMRS}. If we were to modify the definition of relatively $\Delta^0_\alpha$ categorical to be on a cone, it would be equivalent to Definition \ref{def:cconacone}. That is, on a cone, there is no difference between relatively $\Delta^0_\alpha$ categorical and $\Delta^0_\alpha$ categorical.

The notion of relatively $\Delta^0_\alpha$ categoricity is intimately related to that of a Scott family.

\begin{notation} All formulas in this paper will be infinitary formulas, that is, formulas in $L_{\omega_1, \omega}$. See Chapter 6 of \cite{AshKnight00} for background on infinitary formulas and computable infinitary formulas. We will denote by $\Sigma_\alpha^{\inf}$ the infinitary $\Sigma_\alpha$ formulas and by $\Sigma_\alpha^{\comp}$ the computable $\Sigma_\alpha$ formulas.
\end{notation}

\begin{definition} A \emph{Scott family} for a structure $\scrA$ is a countable family $\Phi$ of formulas over a finite parameter such that
\begin{itemize}
\item for each $\abar \in \scrA$, there exists $\varphi \in \Phi$ such that $\scrA \models \varphi(\abar)$
\item if $\varphi \in \Phi$, $\scrA \models \varphi(\abar)$, and $\scrA \models \varphi(\bbar)$, then there is an automorphism of $\scrA$ taking $\abar$ to $\bbar$.
    \end{itemize}
\end{definition}

It follows from work of Scott (see \cite{AshKnight00}) that every countable structure has a Scott family consisting of $\Sigma^{\inf}_\alpha$ formulas for some countable ordinal $\alpha$.

\begin{theorem}[Ash \cite{AshCatinhypdegrees}, see \cite{AshKnight00}]\label{thm:catandScott} A computable structure $\scrA$ is relatively $\Delta^0_\alpha$ categorical if and only if it has a Scott family which is a c.e.\ set of $\Sigma^{\comp}_\alpha$ formulas.
\end{theorem}

Now we can see the power of working on a cone.

\begin{remark}\label{rem:allstructuresalphacatoncone}
Let $\scrA$ be a countable structure. Then $\scrA$ has a Scott family consisting of $\Sigma^{\inf}_\alpha$ formulas for some countable ordinal $\alpha$. Let ${\bf d}$ be such that $\scrA$ and $\alpha$ are ${\bf d}$-computable, and such that the Scott family for $\scrA$ is \ce and consists of $\Sigma^{\comp}_\alpha$ formulas relative to ${\bf d}$. Then $\scrA$ is $\Delta^0_\alpha$ categorical on the cone above ${\bf d}$. That is, every countable structure is $\Delta^0_\alpha$ categorical on a cone for a some $\alpha$.
\end{remark}

We now give the definitions needed to discuss degrees of categoricity on a cone.

\begin{definition} The structure $\scrA$ has \emph{degree of categoricity $\textup{\textbf{d}}$ relative to $\textup{\textbf{c}}$} if $\textbf{d}$ can compute an isomorphism between any two $\textbf{c}$-computable copies of $\scrA$, and moreover $\textbf{d} \geq \textbf{c}$ is the least degree with this property. If in addition to this there exist two $\textbf{c}$-computable copies of $\scrA$ such that for every isomorphism $f$ between them, $f \oplus \textbf{c} \geq_T \textbf{d}$, then we say $\scrA$ has \emph{strong degree of categoricity $\textbf{d}$ relative to $\textbf{c}$}.
\end{definition}

\begin{definition}We say that a structure $\scrA$ has a \emph{(strong) degree of categoricity on a cone}, if there is some $\textbf{d}$ such that for every $\textbf{c} \geq \textbf{d}$, $\scrA$ has a (strong) degree of categoricity relative to \bc.
\end{definition}

\begin{definition} We say that a structure $\scrA$ has \emph{(strong) degree of categoricity ${\bf 0^{(\alpha)}}$ on a cone}, if there is some $\textbf{d}$ such that for every $\textbf{c} \geq \textbf{d}$, $\scrA$ has (strong) degree of categoricity ${\bf c^{(\alpha)}}$ relative to \bc.
\end{definition}

\section{Isomorphism of C.E.\ Degree}\label{sec:ce-degree}

Theorem \ref{thm:main2} follows from the following more technical statement.

\begin{thm}\label{thm:ce-degree}
Let $\mc{A}$ be a structure. Suppose that $\mc{A}$ is $\Delta^0_\alpha$ categorical on a cone. Then there is a degree \textup{\textbf{c}} such that for every copy $\mc{B}$ of $\mc{A}$, there is a degree \textup{\textbf{d}} that is $\Sigma^0_{\alpha - 1}$ in and above $\mc{B} \oplus \textup{\textbf{c}}$ if $\alpha$ is a successor ordinal, or $\Delta^0_\alpha$ in and above $\mc{B} \oplus \textup{\textbf{c}}$ if $\alpha$ is a limit ordinal, such that
\begin{enumerate}
	\item \textup{\textbf{d}} computes some isomorphism between $\mc{A}$ and $\mc{B}$ and
	\item for every isomorphism $f$ between $\mc{A}$ and $\mc{B}$, $f \oplus \textup{\textbf{c}} \geq_T \textup{\textbf{d}}$.
\end{enumerate}
\end{thm}

Before giving the proof, we consider two motivating examples.

\begin{exam}
Let $\mc{N}$ be the standard presentation of $(\omega,<)$. If $\mc{A}$ is any other presentation, let $\Succ(\mc{A})$ be the successor relation in $\mc{A}$. Then the unique isomorphism between $\mc{N}$ and $\mc{A}$ has the same Turing degree as $\Succ(\mc{A})$. Note that $\Succ(\mc{A})$ is $\Pi^0_1$.
\end{exam}

\begin{exam}
Let $\mc{V}$ be an infinite-dimensional $\mathbb{Q}$-vector space with a computable basis. If $\mc{W}$ is any other presentation of $\mc{V}$, let $\Indep(\mc{W})$ be the independence relation in $\mc{W}$. Then any isomorphism between $\mc{V}$ and $\mc{W}$ computes $\Indep(\mc{W})$, and $\Indep(\mc{W})$ computes a basis for $\mc{W}$ and hence an isomorphism between $\mc{V}$ and $\mc{W}$. Note that $\Indep(\mc{W})$ is $\Pi^0_1$.
\end{exam}

Theorem \ref{thm:main2} says that this is the general situation for natural structures.

\begin{proof}[Proof of Theorem \ref{thm:ce-degree}]
Let \textbf{c} be a degree such that $\mc{A}$ is \textbf{c}-computable and $\Delta^0_\alpha$-categorical on the cone above \textbf{c}. By increasing \textbf{c} to absorb the effectiveness conditions of Theorem 17.7 and Proposition 17.6 of \cite{AshKnight00}, $\scrA$ has a c.e.\ Scott family $S$ consisting of $\Sigma^{\comp}_\alpha$ formulas relative to \textbf{c}. Increasing \textbf{c}, we may assume that $S$ consists of formulas of the form $(\exists \bar{x}) \varphi$ where $\varphi$ is $\Pi^{\comp}_\beta$ relative to \textbf{c} for some $\beta < \alpha$. Further increasing \textbf{c}, we may assume that \textbf{c} can decide whether two formulas from $S$ are satisfied by the same elements. Then we can replace $S$ by a Scott family in which every tuple from $\mc{A}$ satisfies a unique formula from $S$. Finally, by replacing \textbf{c} with a higher degree, we may assume that \textbf{c} can compute, for an element of $\mc{A}$, the unique formula of $S$ which it satisfies, and can decide, for each tuple of the appropriate arity, whether or not it is a witness to the existential quantifier in that formula. This is the degree \textbf{c} from the statement of the theorem.

Let $\mc{B}$ be a copy of $\mc{A}$. 
Consider the set \[ S(\mc{B}) = \{ (\bar{b},\varphi) : \mc{B} \models \varphi(\bar{b}), \varphi \in S\}. \]
Let \textbf{d} be the degree of $S(\scrB) \join \scrB \join \textbf{c}$.
First, note that the set \[ S(\mc{A}) = \{ (\bar{a},\varphi) : \mc{A} \models \varphi(\bar{a}), \varphi \in S\} \] is \textbf{c}-computable. If $f$ is an isomorphism between $\mc{A}$ and $\mc{B}$, then $f \oplus \textbf{c}$ computes $S(\mc{A})$. Then using $f$ and $S(\mc{A})$, we can compute $S(\mc{B})$. Thus
\[ f \oplus \textbf{c} \geq_T S(\mc{B}) \oplus \mc{B} \oplus \textbf{c} \equiv_T \textbf{d} \] for every isomorphism $f$ between $\mc{A}$ and $\mc{B}$.

On the other hand, $\textbf{c}$ computes $S(\mc{A})$. Using $S(\mc{B})$ and $S(\mc{A})$ we can compute an isomorphism between $\mc{A}$ and $\mc{B}$. So there is an isomorphism $f$ between $\mc{A}$ and $\mc{B}$ such that
\[ f \oplus \textbf{c} \equiv_T S(\mc{B}) \oplus \mc{B} \oplus \textbf{c} \equiv_T \textbf{d}. \]

We now introduce a related set $T(\mc{B})$. We will show that $T(\mc{B}) \oplus \mc{B} \oplus \textbf{c} \equiv_T \textbf{d}$. If $\alpha$ is a successor ordinal, then $T(\mc{B})$ will be $\Pi^0_{\alpha-1}$ in $\mc{B}\oplus \textbf{c}$, and if $\alpha$ is a limit ordinal then $T(\mc{B})$ will be $\Delta^0_{\alpha}$ in $\mc{B}\oplus \textbf{c}$. Thus \textbf{d} will be a degree of the appropriate type. We may consider the elements of $\mc{B}$ to be ordered, and hence order tuples from $\mc{B}$ via the lexicographic order. Let $T(\mc{B})$ be the set of tuples $(\bar{a},\bar{b},\varphi)$ where:
\begin{enumerate}
	\item $\varphi(\bar{x},\bar{y})$ is a \textbf{c}-computable $\Pi^0_\beta$ formula, for some $\beta < \alpha$,
	\item $(\exists \bar{y}) \varphi(\bar{x},\bar{y})$ is in $S$, and
	\item $\mc{B} \models \varphi(\bar{a},\bar{c})$, for some $\bar{c} \leq \bar{b}$ in the lexicographical ordering of tuples from $\scrB$.
\end{enumerate}

It is easy to see that if $\alpha$ is a successor ordinal, then $T(\mc{B})$ is $\Pi^0_{\alpha-1}$ in $\mc{B}\oplus \textbf{c}$, and if $\alpha$ is a limit ordinal then $T(\mc{B})$ is $\Delta^0_{\alpha}$ in $\mc{B}\oplus \textbf{c}$. Now we will argue that $T(\mc{B}) \oplus \mc{B} \oplus \textbf{c} \equiv_T S(\mc{B}) \oplus \mc{B} \oplus \textbf{c}$.

Suppose we want to check whether $(\bar{a},\bar{b},\varphi) \in T(\mc{B})$ using $S(\mc{B}) \oplus \mc{B} \oplus \textbf{c}$. Using \textbf{c}, we first compute whether (1) and (2) hold for $\varphi$. Then using $S(\mc{B}) \oplus \mc{B} \oplus \textbf{c}$ we can compute an isomorphism $f \colon \mc{B} \to \mc{A}$. Now for each $\bar{c} \leq \bar{b}$ in $\scrB$, $\scrB \models \varphi(\bar{a}, \bar{c})$ if and only if $\scrA \models \varphi(f(\bar{a}), f(\bar{c}))$. In $\scrA$, using \textbf{c} we can decide whether $\scrA \models \varphi(f(\bar{a}), f(\bar{c}))$.

On the other hand, to see whether $(\bar{a},(\exists \bar{y}) \varphi(\bar{x},\bar{y}))$ is in $S(\mc{B})$ using $T(\mc{B})$, look for $\bar{b}$ and $\psi$ such that $(\bar{a},\bar{b},\psi) \in T(\mc{B})$. Some such $\psi$ and witness $\bar{b}$ must exist, since $\bar{a}$ satisfies some formula from $S$. Then $(\bar{a},(\exists \bar{y}) \varphi(\bar{x},\bar{y})) \in S(\mc{B})$ if and only if $\varphi = \psi$ (recall that we assumed that each element of $\mc{A}$ satisfied a unique formula from the Scott family).
\end{proof}

\section{Not computably categorical on any cone} \label{sec:notcconanycone}

This section is devoted to the proof of Theorem \ref{thm:main1} for structures which are $\bf{0'}$-computably categorical on a cone. The general case of the theorem will require the $\eta$-systems developed in the next section, and will be significantly more complicated, so the proof of this simpler case should be helpful in following the proof in the general case, and in fact, we have a slightly stronger theorem in this case. We first recall some definitions from \cite{AshKnight00}.


\begin{definition}[Back-and-forth relations] For a structure $\scrA$ tuples $\bar{a},\bar{b} \in \scrA$ of the same length
\begin{itemize}
\item $\bar{a} \leq_0 \bar{b}$ if and only if for every quantifier-free formula $\varphi(\bar{x})$ with G\"odel number less than $length(\bar{a})$, if $\scrA \models \varphi(\bar{a})$ then $\scrB \models \varphi(\bar{b})$.
\item for $\alpha > 0$, $\bar{a} \leq_\alpha \bar{b}$ if and only if, for each $\bar{d}$ in $\scrA$ and each $0 \leq \beta < \alpha$, there exists $\bar{c}$ in $\scrA$ such that $\bar{b}, \bar{d} \leq_\beta \bar{a}, \bar{c}$.
\end{itemize}
\end{definition}

\begin{definition}[p. 269 Ash-Knight] For tuples $\cbar$ and $\abar$ in $\scrA$, we say that $\abar$ is $\alpha$-free over $\cbar$ if for any $\abar_1$ and for any $\beta < \alpha$, there exist $\abar'$ and $\abar'_1$ such that $\cbar,\abar,\abar_1 \leq_\beta \cbar, \abar', \abar_1'$ and $\cbar, \abar' \nleq_\alpha \cbar, \abar$.
\end{definition}

\begin{definition}[p. 241 Ash-Knight] A structure $\mc{A}$ is \textit{$\alpha$-friendly} if for $\beta < \alpha$, the standard back-and-forth relations $\leq_\beta$ are c.e.\ uniformly in $\beta$.
\end{definition}

There is a version of Theorem \ref{thm:catandScott} for the non-relative notion of categoricity. The following result is the part of it that we will need.

\begin{proposition}[Prop 17.6 from \cite{AshKnight00}] Let $\scrA$ be a computable structure. Suppose $\scrA$ is $\alpha$-friendly, with computable existential diagram. Suppose that there is a tuple $\cbar$ in $\scrA$ over which no tuple $\abar$ is $\alpha$-free. Then $\scrA$ has a formally $\Sigma^0_\alpha$ Scott family, with parameters $\cbar$.
\end{proposition}

\begin{corollary}\label{cor:not-cat gives freeness} Suppose that $\scrA$ is not $\Delta^0_\alpha$ categorical on any cone. Then for any $\cbar$ in $\scrA$, there is some $\abar \in \scrA$ that is $\alpha$-free over $\cbar$.
\end{corollary}

\begin{thm}\label{thm:easy-version} Let $\scrA$ be a structure. If $\scrA$ is not computably categorical on any cone, then there exists an $\textup{\be}$ such that for all $\textup{\bd} \geq \textup{\be}$, if $\textup{\bc}$ is \textup{\ce} in and above \textup{\bd}, then there exists a \textup{\bd}-computable copy $\scrB$ of $\scrA$ such that
\begin{enumerate}
	\item there is a \textup{\bc}-computable isomorphism between $\scrA$ and $\scrB$ and
	\item for every isomorphism $f$ between $\scrA$ and $\scrB$, $f \oplus \textup{\bd}$ computes $\textup{\bc}$.
\end{enumerate}
\end{thm}
\begin{proof} Suppose $\scrA$ is not computably categorical on any cone. Before we begin, note that since $\scrA$ is not computably categorical on any cone, for any tuple $\cbar$ in $\scrA$, there exist a tuple $\abar$ in $\scrA$ that is $1$-free over $\cbar$. Let $\be$ be such that:
\begin{enumerate}
\item $\scrA$ is $\be$-computable,
	\item $\be$ computes a Scott family for $\scrA$ where each tuple satisfies a unique formula, and $\be$ can compute which which formula a tuple of $\scrA$ satisfies,
	\item $\scrA$ is 1-friendly relative to \be, and
	\item given $\cbar$, $\be$ can compute the least tuple $\abar$ that is $1$-free over $\cbar$.
\end{enumerate}
Let $\bd \geq \be$, and let $\bc$ be \ce in and above \bd. Let $C \in \bc$ be such that we have a $\bd$-computable approximation to $C$ where at most one number is enumerated at each stage, and there are infinitely many stages when nothing is enumerated.

We will build $\mc{B}$ with domain $\omega$ by a \bd-computable construction. We will build a bijection $f \colon \omega \to \scrA$ and $\scrB$ will be the pullback, along $f$, of $\scrA$. At each stage $s$, we will have a finite approximation $f_s$ to $f$, and $\scrB[s]$ a finite part of the diagram of $\mc{B}$ so that $f_s$ is a partial isomorphism between $\mc{B}[s]$ and $\mc{A}$. Once we put something into the diagram of $\mc{B}$, we will not remove it, and so $\mc{B}$ will be \bd-computable. While the approximation $f_s$ will be \bd-computable, $f$ will be $C$-computable.

We will have distinguished tuples $\abar_0 \in \scrA$ and $\bbar_0 \in \scrB$, such that for any isomorphism $g: \scrB \rt \scrA$, we will have $0 \nin C$ if and only if $g(\bbar_0)$ is automorphic to $\abar_0$ in $\scrA$. For $n > 0$ the strategy for coding whether $n \in C$ will be the same, but our $\abar_n$ and $\bbar_n$ will be re-defined each time some $m < n$ is enumerated into $C$. When $n$ is enumerated into $C$, we will be able to redefine $f$ on $\bbar_n$ and on all greater values. At each stage $s$, we have current approximations $\abar_n[s]$ and $\bbar_n[s]$ to these values. The tuple $\bbar_n[s]$ will be a series of consecutive elements of $\omega$; by $\scrB \Eres \bbar$ we mean the elements of $\scrB$ up to, and including, those of $\bbar$, and by $\scrB \res \bbar$ we mean those up to, but not including, $\bbar$.

At each stage, if $n \notin C$, for those $\abar_n$ and $\bbar_n$ which are defined at that stage we will have $f(\bbar_n)$ is $1$-free over $f(\scrB \res \bbar_n)$; otherwise, we will have $f(\scrB \res \bbar_n) f(\bbar_n) \niso f(\scrB \res \bbar_n) \abar_n$.

\vspace{5pt}

\noindent \textit{Construction.}

\vspace{5pt}

\textit{Stage 0}: Let $\abar_0[0]$ be the least tuple of $\scrA$ that is $1$-free, and let $\bbar_0[0]$ be the first $|\abar_0|$-many elements of $\omega$. Define $f_0$ to be the map $\bbar_0[0] \mapsto \abar_0[0]$. Let $\scrB[0]$ be the pullback, along $f_0$, of $\scrA$, using only the first $|\abar_0[0]|$-many symbols from the language.

\textit{Stage $s+1$}: Suppose $n$ enters $C$ at stage $s+1$. Let $\bbar = \scrB[s]\res\bbar_{n}[s]$. Let $\bbar'$ be those elements of $\scrB[s]$ which are not in $\bbar$ or $\bbar_n[s]$. Then, since $\abar_n[s]$ is $1$-free over $f(\bbar)$, there are $\abar,\abar' \in \scrA$ such that
\[ f(\bbar),\abar_n[s],f(\bbar') \leq_0 f(\bbar),\abar,\abar', \text{ but } f(\bbar), \abar \not\iso f(\bbar), \abar_n[s].\]
Define $f_{s+1}$ to map $\bbar,\bbar_n[s],\bbar'$ to $f(\bbar),\abar,\abar'$. For $m \leq n$, let $\abar_m[s+1] = \abar_m[s]$ and $\bbar_m[s+1] = \bbar_m[s]$. For $m > n$, $\abar_m[s+1]$ and $\bbar_m[s+1]$ are undefined.

If nothing enters $C$ at stage $s+1$, let $n$ be least such that $\abar_n[s]$ is undefined. For $m < n$, let $\abar_m[s+1] = \abar_m[s]$ and $\bbar_m[s+1] = \bbar_m[s]$. Let $\abar_{n}[s+1]$ be the least tuple that is $1$-free over $\ran(f_{s})$. Extend $f_s$ to $f_{s+1}$ with range $\scrA\Eres\abar_{n}[s+1]$ by first mapping new elements $\bbar_{n}[s+1]$ of $\omega$ to $\abar_{n}[s+1]$, and then mapping more elements to the rest of $\scrA\Eres\abar_{n}[s+1]$. If $n \in C$, we must modify $f_{s+1}$ as described above in the case $n$ entered $C$.

In all cases, let $\mc{B}[s+1]$ be the pullback, along $f_{s+1}$, of $\mc{A}$. We have $\mc{B}[s] \subseteq \mc{B}[s+1]$.

\vspace{5pt}

\noindent \textit{End of construction.}

\vspace{5pt}

Since $\abar_n$ and $\bbar_n$ are only re-defined when there is an enumeration of some $m \leq n$ into $C$, it is easy to see that for each $n$, $\abar_n$ and $\bbar_n$ eventually reach a limit. Moreover, since the $\abar_n$ and $\bbar_n$ form infinite sequences in $\scrA$ and $\scrB$, respectively, and since $f$ is not re-defined on $\scrB\Eres \bbar_n$ unless there is an enumeration of $m \leq n$ into $C$, we see that $f$ is an isomorphism between $\scrB$ and $\scrA$. Moreover, $C$ can compute a stage when $\abar_n$ and $\bbar_n$ have reached their limit, and hence $f$ is \bc-computable.

Now suppose $g: \scrB \rt \scrA$ is an isomorphism. To compute $C$ from $g \oplus \bd$, proceed as follows. Compute $g(\bbar_0)$. Ask $\bd$ whether $(\scrA, g(\bbar_0)) \iso (\scrA, \abar_0)$. If yes, then $0 \nin C$. We also know that $\bbar_1 = \bbar_1[0]$ and that $\abar_1 = \abar_1 [0]$. If $(\scrA, g(\bbar_0)) \niso (\scrA, \abar_0)$, then $0 \in C$. Compute $s$ such that $0 \in C[s]$. Then $\bbar_1 = \bbar_1[s]$ and $\abar_1 = \abar_1[s]$. Continuing in this way, given $\bbar_{n}$ and $\abar_{n}$, we ask $\bd$ whether $(\scrA, g(\bbar_{n})) \iso (\scrA, \abar_{n})$, using the answer to decide whether $n \in C$ and to compute $\bbar_{n+1}$ and $\abar_{n+1}$.
\end{proof}

Using Knight's theorem on the upwards closure of degree spectra \cite{Knight86}, we get a slight strengthening of the above theorem.

\begin{cor} Let $\scrA$ be a computable structure. If $\scrA$ is not computably categorical on any cone, then there exists an $\textup{\be}$ such that for all $\textup{\bd} \geq \textup{\be}$, if $\textup{\bc}$ is \textup{\ce} in and above \textup{\bd}, then there exists a \textup{\bd}-computable copy $\scrB$ of $\scrA$ such that every isomorphism between $\scrA$ and $\scrB$ computes \textup{\bc}, and such that there exists a \textup{\bc}-computable isomorphism between $\scrA$ and $\scrB$.
\end{cor}
\begin{proof} Take $\be$ as guaranteed by the theorem, and fix $\bd \geq \be$, and let $\bc$ be \ce in $\bd$. Let $\scrC$ be as guaranteed by the theorem. Since $\scrC$ is \bd-computable, by the proof of Knight's upward closure theorem \cite{Knight86} (and noting that a ``trivial'' structure is computably categorical on a cone), there exists $\scrB$ such that deg$(\scrB)=\bd$ and such that there exists a $\bd$-computable isomorphism $h: \scrB \iso \scrC$. Now since $\scrA$ is computable and deg$(\scrB)=\bd$, any isomorphism $g: \scrA \iso \scrB$ computes $\bd$. Since $\bd$ computes $h$, $g$ computes the isomorphism $g \circ h :\scrC \iso \scrA$ and hence it computes $\bc$. On the other hand, since $\bc$ computes $\bd$ and hence $h$, and since $\bc$ computes an isomorphism between $\scrA$ and $\scrC$, we have that $\bc$ computes an isomorphism between $\scrA$ and $\scrB$.
\end{proof}

\begin{cor}
On a cone, a structure cannot have degree of categoricity strictly between ${\bf 0}$ and ${\bf 0'}$. That is, if $\scrA$ is not computably categorical on any cone, and if $\scrA$ has a degree of categoricity on a cone, then there is some $\textup{\be}$ such that for all $\textup{\bd} \geq \textup{\be}$, the degree of categoricity of $\scrA$ relative to $\textup{\bd}$ is at least $\textbf{\textup{d}}'$.
\end{cor}

\begin{cor}
If $\scrA$ is $\Delta^0_2$ categorical on a cone then $\scrA$ has degree of categoricity ${\bf 0}$ or ${\bf 0'}$ on a cone.
\end{cor}

%

\section{A Version of Ash's Metatheorem} \label{sec:metatheorem}

The goal of the remainder of the paper is to prove Theorem \ref{thm:main1}. Our main tool will be a version of Ash's metatheorem for priority constructions which was first introduced in \cite{Ash86a,Ash86b,Ash90}. Ash and Knight's book \cite{AshKnight00} is a good reference. Montalb\'an \cite{Montalban} has recently developed a variant of Ash's metatheorem using computable approximations. Montalb\'an's formulation of the metatheorem also provides more control over the construction; for the proof of Theorem \ref{thm:main1}, we will require this extra control. However, Montalb\'an's version of the metatheorem, as written, only covers ${\bf 0}^{(\eta)}$-priority constructions for $\eta$ a successor ordinal. In this section, we will introduce the metatheorem and expand it to include the case of limit ordinals.

Fix a computable ordinal $\eta$ for which we will define $\eta$-systems and the metatheorem for constructions guessing at a $\Delta^0_\eta$-complete function. Here our notation differs from Montalb\'an's but corresponds to Ash's original notation. What we call an $\eta$-system corresponds to what Ash would have called an $\eta$-system, but what Montalb\'an calls an $\eta$-system we will call an $\eta+1$-system. This will allow us to consider, for limit ordinals $\eta$, what Montalb\'an might have called a $< \eta$-system.

\subsection{\texorpdfstring{Some $\Delta^0_\xi$-complete functions, their approximations, and true stages}{Some complete functions, their approximations, and true stages}}

Before defining an $\eta$-system and stating the metatheorem, we discuss some $\Delta^0_\xi$-complete functions and their approximations as introduced by Montalb\'an \cite{Montalban}. We will introduce orderings on $\omega$ to keep track of our beliefs on the correctness of the approximations.

For each computable ordinal $\xi \leq \eta$, Montalb\'an defines a $\Delta^0_\xi$-complete function $\nabla^{\xi} \in \omega^\omega$, and for each stage $s \in \omega$ a computable approximation $\nabla^{\xi}_s$ to $\nabla^{\xi}$. $\nabla^{\xi}_s$ is a finite string which guesses at an initial segment of $\nabla^{\xi}$. The approximations are all uniformly computable in both $s$ and $\xi$. Montalb\'an shows that the approximation has the following properties (see Lemmas 7.3, 7.4, and 7.5 of \cite{Montalban}):
\begin{enumerate}
	\item[(N1)] For every $\xi$, the sequence of stages $t_0 < t_1 < t_2 < \cdots$ for which $\nabla^\xi_t$ is correct is an infinite sequence with $\nabla^\xi_{t_0} \subseteq \nabla^\xi_{t_1} \subseteq \cdots$ and $\bigcup_{i \in \omega} \nabla_{t_i}^\xi = \nabla^\xi$.
	\item[(N2)] For each stage $s$, there are only finitely many $\xi$ with $\nabla^{\xi}_s \neq \la \ra$, and these $\xi$s can be computed uniformly in $s$.
	\item[(N3)] If $\gamma \leq \xi$, $s\leq t$, and $\langle \rangle \neq \nabla_s^{\xi} \subseteq \nabla^{\xi}_t$, then $\nabla_s^\gamma \subseteq \nabla_t^\gamma$.
\end{enumerate}
We say that $s$ is a \emph{true stage} or \emph{$\eta$-true stage} if $\nabla^\eta_s \subset \nabla^\eta$.

Montalb\'an defines relations $(\leq_\xi)_{\xi < \eta}$ on $\omega$, to be thought of as a relation on stages in an approximation. We will define relations $(\leq_\xi)_{\xi < \eta}$ which are almost, but not exactly, the same as Montalb\'an's (we leave the definition of these relations, and the proofs of their properties, to Lemma \ref{lem:rel-props}). An instance $s \leq_\xi t$ of the relation should be interpreted as saying that, from the point of view of $t$, $s$ is a $\xi$-true stage. A relation $s \leq_\xi t$ is almost, but not exactly, equivalent to saying that for all $\gamma \leq \xi + 1$, $\nabla_s^{\gamma} \subseteq \nabla_t^{\gamma}$. The problem is that we require the property (B4) below.

We also have a computable relation $\trianglelefteq$ on stages with $s \trianglelefteq t$ if and only if, for all $\xi < \eta$, $\nabla_s^{\xi+1} \subseteq \nabla^{\xi+1}_t$. We can interpret $s \trianglelefteq t$ as saying that $s$ appears to be a true stage (or $\eta$-true stage) from stage $t$. This relation is computable by (N2) above.

We will see that the relations $\leq_\xi$ satisfy the following properties:
\begin{enumerate}
	\item[(B0)] $\leq_0$ is the standard ordering on $\omega$.
	\item[(B1)] The relations $\leq_\xi$ are uniformly computable.
	\item[(B2)] Each $\leq_\xi$ is a \textit{preordering} (i.e., reflexive and transitive).
	\item[(B3)] The sequence of relations is \textit{nested} (i.e., if $\gamma \leq \xi$ and $s \leq_\xi t$, then $s \leq_\gamma t$).
	\item[(B4)] The sequence of relations is \textit{continuous} (i.e., if $\lambda$ is a limit ordinal, then $\leq_\lambda = \bigcap_{\xi < \lambda} \leq_\xi$).
	\item[(B5)] For every $s < t$ in $\omega$, if $s \leq_\xi t$ then $\nabla_s^{\xi + 1} \subseteq \nabla_t^{\xi+1}$.
	\item[(B6)] The sequence $t_0 <t_1 <\ldots$ of true stages satisfies $t_0 \trianglelefteq t_1 \trianglelefteq \cdots$ and $\bigcup_{i \in \omega} \nabla_{t_i}^{\eta} = \nabla^{\eta}$. We call the sequence of true stages the \emph{true path}.
	\item[(B7)] For $s \in \omega$, we can compute $H(s)= \max\{\xi < \eta \mid \nabla_s^\xi \not= \la \ra \}$. $H(s)$ has the property that if $t > s$ and $s \ntrianglelefteq t$, then $s \nleq_{H(s)} t$. We call $H(s)$ the \emph{height} of $s$.
	\item[(B8)] For every $\xi$ with $\xi < \eta$, and $r < s < t$, if $r \leq_{\xi} t$ and $s \leq_{\xi} t$, then $r \leq_{\xi} s$. Moreover, if $\xi$ is a successor ordinal, then it suffices to assume that $s \leq_{\xi - 1} t$.
	\item[(B9)] $s \trianglelefteq t$ if and only if for all $\xi < \eta$, $s \leq_\xi t$.
    \item[(B10)] If $t$ is a true stage and $s \trianglelefteq t$, then $s$ is also a true stage.
\end{enumerate}
 Properties (B0)-(B5) are as in Montalb\'an \cite{Montalban}. Our (B6) is a modification of Montalb\'an's (B6). (B7), (B9) and (B10) are new properties. (B8) is Montalb\'an's ($\clubsuit$) together with his Observation 2.1.

We will define, for convenience, the relations $\trianglelefteq_\xi$ for $\xi < \eta$. Let $s \trianglelefteq_\xi t$ if for all $\gamma \leq \xi + 1$, $\nabla_s^\gamma \subseteq \nabla_t^\gamma$. Note that $s \trianglelefteq t$ if and only if for all $\xi < \eta$, $s \trianglelefteq_\xi t$.

\begin{lemma}\label{lem:rel-props}
There is a sequence $(\leq_\xi)_{\xi < \eta}$ satisfying (B0)-(B10).
\end{lemma}
\begin{proof}
The proof of this Lemma is very similar to the proof of Lemma 7.8 of \cite{Montalban}. The definition of our relations $\leq_\xi$ is the same as Montalb\'an's, except for one small change. Let $C$ be the set of tuples $(\lambda,u,v)$ where $\lambda$ is a limit ordinal, $\nabla^{\lambda}_u \subsetneq \nabla^{\lambda}_v$, $\nabla_u^{\lambda + 1} \nsubseteq \nabla_v^{\lambda + 1}$, and if there is $r$ with $\nabla_u^\lambda \subsetneq \nabla_r^\lambda \subsetneq \nabla_v^\lambda$ then $\nabla_u^{\lambda + 1} \subseteq \nabla_r^{\lambda + 1}$. Let $\gamma_{\lambda,v}$ be such that the last entry of $\nabla_v^\lambda$ is $\nabla_v^{\gamma_{\lambda,v}}(0)$.

For $\xi < \eta$, define
\[ s \leq_\xi t \Leftrightarrow s \trianglelefteq_\xi t \text{ and } \neg \exists (\lambda,u,v) \in C (\lambda < \eta \text{ and } \gamma_{\lambda,v} < \xi \text{ and } u \leq s < v \trianglelefteq_{\gamma_{\lambda,v}} t). \]
The only difference between our definition and Montalb\'an's is that we require $\lambda < \eta$.

The proof that $(\leq_\xi)_{\xi < \eta}$ satisfies (B0)-(B5) is the same as the proof of Lemma 7.8 of \cite{Montalban}. (B6) follows from (N1) and the fact that if $r < t$ are both true stages, then $r \trianglelefteq s$. Indeed, if $r < t$ are both true stages, then $\nabla^\eta_r \subseteq \nabla^\eta_s$, so by Lemmas 7.5 and 7.6 in \cite{Montalban}, $\nabla_{r}^{\xi + 1} \subseteq \nabla_{s}^{\xi + 1}$ for all $\xi < \eta$, so that $r \trianglelefteq s$. (B7) follows from the fact that if $s \not\trianglelefteq_\xi t$, then $\nabla^\xi_s \not= \la \ra$, and for each $s$, there are only finitely many $\xi$ with $\nabla^\xi_s \not= \la \ra$, which are uniformly computable in $s$. (B8) follows from the proof of Montalb\'an's ($\clubsuit$) in Lemma 7.8 together with his Observation 2.1\footnote{Montalb\'an uses, in the proof of ($\clubsuit$), Observation 2.1 for the relations $(\trianglelefteq_\xi)_{\xi < \eta}$. The given proof of Observation 2.1 requires continuity (B4), which $(\trianglelefteq_\xi)_{\xi < \eta}$ does not have. However, it is easy to see that the proof of Lemma 7.7 suffices to prove Observation 2.1 for these relations.}. The proof of (B6) in Lemma 7.8 of \cite{Montalban} also suffices to prove (B9). (B10) is immediate from the definitions if $\eta$ is a successor, and follows from Lemma 7.6 in \cite{Montalban} if $\eta$ is limit.
\end{proof}

\subsection{\texorpdfstring{$\eta$-systems and the metatheorem}{Eta-systems and the metatheorem}}

We are now ready to define an $\eta$-system. The definition is essentially the same as for Montalb\'an, except that what Montalb\'an would have called an $\eta$-system, we call an $\eta+1$-system.
\begin{defn}
An $\eta$-system is a tuple $(L,P,(\leq_\xi^L)_{\xi < \eta},E)$ where:
\begin{enumerate}
	\item $L$ is a c.e.\ subset of $\omega$ called the set of \textit{states}.
	\item $P$ is a c.e.\ subset of $L^{<\omega}$ called the \textit{action tree}.
	\item $(\leq_\xi^L)_{\xi < \eta}$ is a nested sequence of c.e.\ pre-orders on $L$ called the \textit{restraint relations}.
\item  $\ell \trianglelefteq^L \ell'$ is c.e., where we define $\ell \trianglelefteq^L \ell'$ if and only if $\ell \leq^L_\xi \ell'$ for all $\xi < \eta$.
	\item $E \subseteq L \times \omega$ is a c.e.\ set called the \textit{enumeration function}, and is interpreted as $E(l) = \{ k \in \omega : (l,k) \in E\}$. We require that for $\ell_0,\ell_1 \in L$ with $\ell_0 \leq_0^L \ell_1$, $E(\ell_0) \subseteq E(\ell_1)$.
\end{enumerate}

\end{defn}

\begin{defn}
A \textit{0-run} for $(L,P,(\leq_\xi^L)_{\xi < \eta},E)$ is a finite or infinite sequence $\pi = (\ell_0,\ell_1,\ldots)$ which is in $P$ if it is a finite sequence, or is a path through $P$ if it is an infinite sequence, such that for all $s,t < |\pi|$ and $\xi < \eta$,
\[ s \leq_\xi t \Rightarrow \ell_s \leq_\xi^L \ell_t.\]
If $\pi$ is a 0-run, let $E(\pi) = \bigcup_{s < |\pi|} E(\ell_i)$.
\end{defn}

Given an infinite 0-run $\ell_0,\ell_1,\ldots$ of an $\eta$-system $(L,P,(\leq_\xi^L)_{\xi < \eta},E)$, let $t_0 \trianglelefteq t_1 \trianglelefteq t_2 \trianglelefteq \cdots$ be the true stages. Then by the properties of $E$ above, $E(\pi) = \bigcup_{i \in \omega} E(\ell_{t_i})$. So $E(\pi)$ is c.e., but it is determined by the true stages.

Montalb\'an defines an extendability condition and a weak extendability condition. For our extendability condition, we weaken Montalb\'an's extendability condition even further (as well as modifying it slightly to allow limit ordinals). In order to define our extendability condition, we need the following definition.

\begin{definition}\label{defn: associated stages and ordinals} To any stage $s >0$, we effectively associate a sequence of stages and ordinals as follows.

Choose $t^* < s$ greatest such that $t^* \trianglelefteq s$. Some such $t^*$ exists as $0 \trianglelefteq s$. Now for each $\xi < \eta_0$, let $t_\xi < s$ be the largest such that $t_\xi \leq_\xi s$. Note that $t^* \leq t_\xi$ for each $\xi$ as by (B9) $t^* \leq_\xi s$.

There may be infinitely many $\xi < \eta$, but there are only finitely many possible values of $t_\xi$ since they are bounded by $s$. Since the $\leq_\xi$ are nested (B3), if $\gamma \leq \xi < \eta$, then $t_\xi \leq t_\gamma$. Now we will effectively define stages $t^* = s_k < \cdots < s_0 = s-1$ so that $\{s_0,\ldots,s_k\} = \{t_\xi : \xi < \eta\}$ as sets. Let $s_0 = t_0 = s-1$. Suppose that we have defined $s_i$. If $s_i \trianglelefteq s$, then $k = i$ and we are done. Otherwise, let $\xi_i < \eta$ be the greatest such that $s_i = t_{\xi_i}$. By definition of $s_i$, it is of the form $t_\xi$ for some $\xi$. We can find the greatest such by computably searching for $\xi_i$ such that $s_i \leq_{\xi_i} s$ but $s_i \not\leq_{\xi_i +1} s$; some such $\xi_i$ exists since the relations are continuous and nested. Let $s_{i+1} = t_{\xi_i + 1}$. Since $s_i \nleq_{\xi_i + 1} s$, $s_{i+1} < s_i$. This completes the definition of $s_k < \cdots < s_0 =s-1$ and $\xi_0 < \ldots < \xi_{k-1} < \eta$.

By (B8), for $i < k$, since $s_{i+1} \leq_{\xi_i+1} s$ and $s_i \leq_{\xi_i} s$, $s_{i+1} \leq_{\xi_i + 1} s_i$.

\end{definition}
\begin{definition}

We say that an $\eta$-system $(L, P, (\leq_\xi^L)_{\xi \leq \eta}, E)$ satisfies the \emph{extendability condition} if: whenever we have a finite $0$-run $\pi = \la \ell_0, ..., \ell_{s-1} \ra$ such that for all $i < k$, $\ell_{s_{i+1}} \leq_{\xi_i+1}^L \ell_{s_i}$, where $s_k < s_{k-1} < ... < s_0= s-1$ and $\xi_0 < \xi_1 < ... < \xi_{k-1} < \eta$ are the associated sequences of stages and ordinals to $s$ as in Definition \ref{defn: associated stages and ordinals}, then there exists an $\ell \in L$ such that $\pi \conc \ell \in P$, $\ell_{s_k} \trianglelefteq^L \ell$, and for all $i < k, \ell_{s_i} \leq_{\xi_i}^L \ell$.

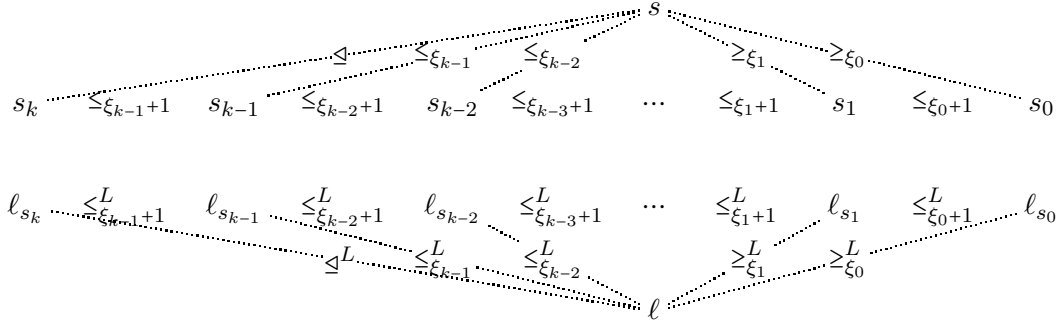
\[ \xymatrix@C=5em{ &  &  & s\ar@{..}[drr]|*[@]{\geq_{\xi_0}}\ar@{..}[dr]|*[@]{\geq_{\xi_1}}\\
s_{k}\ar@{}[r]|*[@]{\leq_{\xi_{k-1}+1}}\ar@{..}[urrr]|*[@]{\trianglelefteq} & s_{k-1}\ar@{}[r]|*[@]{\leq_{\xi_{k-2}+1}}\ar@{..}[urr]|*[@]{\leq_{\xi_{k-1}}} & s_{k-2}\ar@{}[r]|*[@]{\leq_{\xi_{k-3}+1}}\ar@{..}[ur]|*[@]{\leq_{\xi_{k-2}}} &  \cdots\ar@{}[r]|*[@]{\leq_{\xi_{1}+1}} & s_{1}\ar@{}[r]|*[@]{\leq_{\xi_{0}+1}} & s_{0} \\
\ell_{s_{k}}\ar@{}[r]|-*[@]{\leq^L_{\xi_{k-1}+1}}\ar@{..}[drrr]|*[@]{\trianglelefteq^L} & \ell_{s_{k-1}}\ar@{}[r]|-*[@]{\leq^L_{\xi_{k-2}+1}}\ar@{..}[drr]|*[@]{\leq^L_{\xi_{k-1}}} & \ell_{s_{k-2}}\ar@{}[r]|-*[@]{\leq^L_{\xi_{k-3}+1}}\ar@{..}[dr]|*[@]{\leq^L_{\xi_{k-2}}} & \cdots\ar@{}[r]|-*[@]{\leq^L_{\xi_{1}+1}} & \ell_{s_{1}}\ar@{}[r]|-*[@]{\leq^L_{\xi_{0}+1}} & \ell_{s_{0}}\\
 &  &  & \ell\ar@{..}[urr]|*[@]{\geq^L_{\xi_0}}\ar@{..}[ur]|*[@]{\geq^L_{\xi_1}}
}
 \]
\end{definition}

Now we are ready for the metatheorem.

\begin{thm}\label{thm:meta}
For every $\eta$-system $(L,P,(\leq_\xi^L)_{\xi < \eta},E)$ with the extendability condition, there is a computable infinite $0$-run $\pi$. A 0-run can be built uniformly in the $\eta$-system.
\end{thm}

\begin{proof}[Proof of Theorem \ref{thm:meta}]
The proof is essentially the same as the proof of Theorem 3.2 in \cite{Montalban}. By the trivial case of the extendability condition, there is $\ell_0 \in L$ with $\la \ell_0 \ra \in P$. Now suppose that we have a 0-run $\pi = \la \ell_0,\ldots,\ell_{s-1} \ra$. We want to define $\ell_s \in L$ such that $\pi \conc \ell_s \in P$, and such that for every $\xi < \eta$, if $t \leq_\xi s$, then $\ell_t \leq_\xi^L \ell_s$.

Let $\{ t_\xi \mid \xi < \eta \}$, $s_k < \ldots s_0 = s-1$, and $\xi_1 < \ldots < \xi_k$ be as in Definition \ref{defn: associated stages and ordinals}. If $t \leq_\xi s$, then $t \leq t_\xi$, and by (B8), $t \leq_\xi t_\xi$, so since $\pi$ is a $0$-run $\ell_t \leq_\xi^L \ell_{t_\xi}$. So it is sufficient to find $\ell$ with $\pi \conc \ell \in P$ such that, for $\xi < \eta$, $\ell_{t_\xi} \leq_\xi^L \ell$. That is, we must find an $\ell$ with $\pi \conc \ell \in P$, $\ell_{s_k} \trianglelefteq^L \ell$ and $\ell_{s_i} \leq_{\xi_i}^L \ell$ for $0 \leq i <k$.

By (B8), for $i \leq k$, since $s_{i+1} \leq_{\xi_i+1} s$ and $s_i \leq_{\xi_i} s$, $s_{i+1} \leq_{\xi_i + 1} s_i$. Since $p$ is a 0-run, $\ell_{s_{i+1}} \leq^L_{\xi_i + 1} \ell_{s_i}$. By the extendability condition, there is $\ell \in L$ with $p \conc \ell \in P$, $\ell_{s_k} \trianglelefteq^L \ell$, and $\ell_{s_i} \leq^L_{\xi_i} \ell$ for $i < k$. We can find such an $l$ effectively, since we have described how to compute the $s_i$ and since the relations $\leq^L_\xi$ and $\trianglelefteq^L$ are computable.
\end{proof}

\section{Proof of Theorem \ref{thm:main1}} \label{sec:mainthm}

In this section, we will give the proof of Theorem \ref{thm:main1}. To begin, we prove the following lemma which we will use for coding.

\begin{lemma}\label{lem:partial-isos-different}
Let $\bar{x}$ be a tuple. Let $\alpha_1 > \beta_1,\ldots,\alpha_n > \beta_n$ be computable ordinals with $\beta_1 \geq \beta_2 \geq \cdots \geq \beta_n$. Let $\bar{u}_1,\ldots,\bar{u}_n$ and $\bar{v}_1,\ldots,\bar{v}_n$ be tuples such that $|\bar{u}_{i+1}| = |\bar{u}_i| + |\bar{v}_i|$ and such that $\bar{v}_i$ is $\alpha_i$-free over $\bar{u}_i$. Then there is a tuple $\bar{y}$ such that, for each $i = 1,\ldots,n$,
\begin{enumerate}
	\item $\bar{x} \res_{|\bar{u}_1|} = \bar{y}\res_{|\bar{u}_1|}$,
	\item $\bar{x} \res_{|\bar{u}_i| + |\bar{v}_i|} \leq_{\beta_i} \bar{y} \res_{|\bar{u}_i| + |\bar{v}_i|}$,
	\item $\bar{y} \res_{|\bar{u}_i| + |\bar{v}_i|} \ncong \bar{u}_i \bar{v}_i$.
\end{enumerate}
\end{lemma}
\begin{proof}
We will inductively define tuples $\bar{x}_0, \ldots, \bar{x}_n$, so that taking $\bar{y}=\bar{x}_n$ satisfies the lemma.

Begin with $\bar{x}_0 = \bar{x}$, so $\bar{x}_0$ satisfies (1) and (2).

Given $\bar{x}_m$ satisfying (1) and (2) for all $i$, and (3) for $i = 1,\ldots,m$, define $\bar{x}_{m+1}$ as follows. If $\bar{x}_m$ already satisfies (3) for $i=m+1$, set $\bar{x}_{m+1} = \bar{x}_m$. Otherwise, $\bar{x}_m \res_{|\bar{u}_{m+1}| + |\bar{v}_{m+1}|} \cong \bar{u}_{m+1} \bar{v}_{m+1}$. Since $\bar{v}_{m+1}$ is $\alpha_{m+1}$-free over $\bar{u}_{m+1}$, there is $\bar{x}_{m+1}$ with $\bar{x}_m \leq_{\beta_{m+1}} \bar{x}_{m+1}$, $\bar{x}_m \res_{|\bar{u}_{m+1}|} = \bar{x}_{m+1} \res_{|\bar{u}_{m+1}|}$, and $\bar{x}_{m+1} \res_{|\bar{u}_{m+1}| + |\bar{v}_{m+1}|} \ncong \bar{u}_{m+1} \bar{v}_{m+1}$. So $\bar{x}_{m+1}$ satisfies (3) for $i = m+1$. Note that since $\bar{x}_m \res_{|\bar{u}_{m+1}|} = \bar{x}_{m+1} \res_{|\bar{u}_{m+1}|}$, we have $\bar{x}_{m+1} \res_{|\bar{u}_i| + |\bar{v}_i|} = \bar{x}_{m} \res_{|\bar{u}_i| + |\bar{v}_i|}$ for $i \leq m$, so that $\bar{x}_{m+1}$ satisfies (1) and satisfies (2) and (3) for $1 \leq i \leq m$. Since $\bar{x}_m \leq_{\beta_{m+1}} \bar{x}_{m+1}$, and for $i \geq m+1$, $\beta_i \leq \beta_{m+1}$, we have $\bar{x} \res_{|\bar{u}_i| + |\bar{v}_i|} \leq_{\beta_i} \bar{x}_{m} \res_{|\bar{u}_i| + |\bar{v}_i|} \leq_{\beta_i} \bar{x}_{m+1} \res_{|\bar{u}_i| + |\bar{v}_i|}$ for such $i$. So (2) holds for $\bar{x}_{m+1}$.
\end{proof}

Now we are ready to prove the main theorem, Theorem \ref{thm:main1}. The proof will use the $\eta$-systems as developed in the previous section, together with a strategy similar to that in the proof of Theorem \ref{thm:easy-version}. Theorem \ref{thm:main1} will follow easily from the following technical result.

\begin{thm}
Let $\scrA$ be a structure. If $\eta$ is an ordinal and $\scrA$ is not $\Delta^0_\beta$ categorical on any cone for any $\beta < \eta$, then there exists an $\textup{\be}$ such that for all $\textup{\bd} \geq \textup{\be}$, there exists a \textup{\bd}-computable copy $\scrB$ of $\scrA$ such that
\begin{enumerate}
	\item there is a $\Delta^0_\eta(\textup{\bd})$-computable isomorphism between $\scrA$ and $\scrB$ and
	\item for every isomorphism $f$ between $\scrA$ and $\scrB$, $f \oplus \textup{\bd}$ computes $\Delta^0_\eta(\textup{\bd})$.
\end{enumerate}
\end{thm}

\begin{proof}
Suppose $\scrA$ is not $\Delta^0_\beta$ categorical on any cone for any $\beta < \eta$. Let $\be$ be such that:
\begin{enumerate}
	\item $\scrA$ is \be-computable, and \be\ computes a Scott family for $\scrA$ in which each tuple satisfies a unique formula and also computes, for tuples in $\scrA$, which formula in the Scott family they satisfy,
	\item $\scrA$ is $\eta+1$-friendly relative to \be,
	\item given a tuple $\bar{a}$ and $\beta < \eta$, \be\ can decide whether a tuple $\bar{b}$ is $\beta$-free over $\bar{a}$. (Such a tuple is guaranteed to exist by Corollary \ref{cor:not-cat gives freeness} since $\scrA$ is not $\Delta^0_\beta$-categorical on any cone.)
\end{enumerate}
Fix $\bd \geq \be$ and $D \in \bd$. Our argument involves a $D$-computable $\eta$-system. To ease notation, we make no further mention of $D$ (e.g., whenever we write $\nabla^\beta$ we really mean $\nabla^\beta(D)$).

We will define our $\eta$-system. Let $B$ be a computable set of constant symbols not occurring in $A$.
Let $L$ be the set of sequences
\[ \la p; (\bar{a}_0,\bar{b}_0),(\bar{a}_1,\bar{b}_1),\ldots,(\bar{a}_r,\bar{b}_r) \ra \] where:
\begin{enumerate}
	\item[(L1)] $p$ is a finite partial bijection $B \to A$,
	\item[(L2)] $\bar{a}_n,\bar{b}_n \in A$ are tuples with $|\bar{a}_{n+1}| = |\bar{a}_n|+|\bar{b}_n|$,
	\item[(L3)] $|\ran(p)| = |\bar{a}_r|+|\bar{b}_r|$,
	\item[(L4)] $\dom(p)$ and $\ran(p)$ include the first $r$ elements of $B$ and $A$ respectively,
  \item[(L5)] $\bar{b}_{n}$ is $\alpha$-free over $\bar{a}_n$, where $\alpha = \max_{m \leq n} H(m)$ (see (B7)).
\end{enumerate}
Note that (L1)-(L4) are clearly computable, and that (L5) is $\be$-computable by property (3) of $\be$.

If $\ell$ has first coordinate $p$, and $\ell'$ has first coordinate $p'$, then for $\xi < \eta$, we set $\ell \leq^L_\xi \ell'$ if and only if $p \leq_\xi p'$, that is, if and only if $\ran(p) \leq_\xi \ran(p')$ as substructures of $\scrA$ under the usual back-and-forth relations.

Then $(\leq_\xi^L)_{\xi < \eta}$ is nested since the usual back-and-forth relations are, and $(\leq_\xi^L)_{\xi < \eta}$ and $\trianglelefteq^L$ are $\be$-computable by property (2) of $\be$.

Let $P$ consist of the sequences $\ell_0,\ldots,\ell_r$ such that
\begin{enumerate}
	\item[(P1)] if \[ \ell_n = \la p;(\bar{a}_0,\bar{b}_0),(\bar{a}_1,\bar{b}_1),\ldots,(\bar{a}_n,\bar{b}_n) \ra \] then
	\[ \ell_{n + 1} = \la p^*;(\bar{a}_0,\bar{b}_0),(\bar{a}_1,\bar{b}_1),\ldots,(\bar{a}_n,\bar{b}_n),(\bar{a}_{n+1},\bar{b}_{n+1}) \ra \]
	with $\dom(p) \subseteq \dom(p^*)$,
	\item[(P2)] for each $n$, if \[ \ell_n = \la p;(\bar{a}_0,\bar{b}_0),(\bar{a}_1,\bar{b}_1),\ldots,(\bar{a}_n,\bar{b}_n) \ra \] then for each $i$, $\ran(p \res_{|\bar{a}_i| + |\bar{b}_i|}) \cong \bar{a}_i\bar{b}_i$ if and only if $i \trianglelefteq n$,
	\item[(P3)] if $m \trianglelefteq n$, $\ell_m$ has first coordinate $p_m$, and $\ell_n$ has first coordinate $p_n$, then $p_m \subseteq p_n$.
\end{enumerate}

Note that (P1) and (P3) are computable, and that (P2) is $\be$-computable by property (1) of $\be$.

Given \[ \ell_n = \la p;(\bar{a}_0,\bar{b}_0),(\bar{a}_1,\bar{b}_1),\ldots,(\bar{a}_n,\bar{b}_n) \ra, \]let $E(\ell)$ be the partial atomic diagram on $\scrB$ obtained by the pullback along $p$ (using only the first $|p|$ logical symbols).

Note that $E(\ell)$ is computable, and if $\ell_0 \leq_0^L \ell_1$ with first coordinates $p_0$ and $p_1$, respectively, then $p_0 \leq_0 p_1$, so that $E(\ell_0) \subseteq E(\ell_1)$.

Thus we have an $\eta$-system $(L,P,(\leq_\xi^L)_{\xi < \eta},E)$.

\begin{lemma}
The $\eta$-system $(L,P,(\leq_\xi^L)_{\xi < \eta},E)$ has the extendability condition.
\end{lemma}
\begin{proof}
Suppose we have a finite $0$-run $\pi = \la \ell_0, ..., \ell_{s-1} \ra$, and let $s_k < s_{k-1} < ... < s_0= s-1$, and $\xi_0 < \xi_1 < ... < \xi_{k-1} < \eta$ be the associated sequences of stages and ordinals to $s$, as in Definition \ref{defn: associated stages and ordinals}. Suppose that for each $i$, the first coordinate of $\ell_{s_i}$ is $q_{s_i}$.

\begin{claim}
There exists $p \supset q_{s_k}$ such that $q_{s_i} \leq_{\xi_i} p$ for $0\leq i \leq k$.
\end{claim}
\begin{proof}
We construct $p$ inductively as follows. We let $q_{s_0}^* = q_{s_0}$, and for $0 \leq i < k$, let $q_{s_{i+1}}^* \supseteq q_{s_{i+1}}$ be such that $q_{s_{i}}^* \leq_{\xi_i} q_{s_{i+1}}^*$. This is possible since $q_{s_{i+1}} \leq_{\xi_i +1} q_{s_i}$ and since $q_{s_i}^* \supseteq q_{s_i}$. Let $p = q_{s_k}^*$. Then certainly $q_{s_k}^* \leq_{\xi_k}^L p$. As $q_{s_{i}}^* \leq_{\xi_i} q_{s_{i+1}}^*$ and $\xi_i < \xi_{i+1}$, it follows inductively that each $q_{s_i}^* \leq_{\xi_i} p$. Since $q_{s_{i}}^* \supseteq q_{s_{i}}$, we have $q_{s_i} \leq_{\xi_i} p$ as desired.
\end{proof}

Let
\[ \ell_{s_0} = \ell_{s-1} =\la q_{s-1}; (\bar{a}_0,\bar{b}_0),(\bar{a}_1,\bar{b}_1),\ldots,(\bar{a}_{s-1},\bar{b}_{s-1}) \ra.\]

\begin{claim}
There exists $p^* \supset q_{s_k}$ such that $q_{s_i} \leq_{\xi_i} p^*$ for $0 \leq i < k$ and such that $\ran(p^* \res_{|\bar{a}_n|+|\bar{b}_n|}) \ncong \bar{a}_n \bar{b}_n$ for $s_k < n \leq s_0 = s-1$.
\end{claim}
\begin{proof}
Let $p \supset q_{s_k}$ be as in the previous claim. We will use Lemma \ref{lem:partial-isos-different}. Let $\bar{x} = \ran(p)$ and $n = s_0 - s_k$. For $i = 1,\ldots,n$, let $\bar{u}_i = \bar{a}_{s_k + i}$ and $\bar{v}_i = \bar{b}_{s_k + i}$. For $i = 1,\ldots,n$, let $\alpha_i = \max_{1 \leq j \leq s_k + i} H(j)$ and let $\beta_i = \xi_{j}$ where $j$ is such that $s_{j+1} < i \leq s_{j}$. Note that by (L5), $\bar{v}_i$ is $\alpha_i$-free over $\bar{u}_i$ and that $\beta_1 \geq \beta_2 \geq \cdots$. Also, if $s_{j+1} < i \leq s_{j}$, then since $s_{j+1} = t_{\xi_j +1}$, $i \not\leq_{\xi_j +1} s$. So $\alpha_i \geq H(i) \geq \xi_j +1 > \xi_j =\beta_i$. Let $\bar{y}$ be the tuple we get by applying Lemma \ref{lem:partial-isos-different} and let $p^*$ map the domain of $p$ to $\bar{y}$. Then
\[ p^* \res_{|\bar{a}_{s_k}| + |\bar{b}_{s_k}|} = p \res_{|\bar{a}_{s_k}| + |\bar{b}_{s_k}|} \supset q_{s_k} \]
and so $p^* \supseteq q_{s_k}$. Also,
\[ q_{s_i} \leq_{\xi_i} p \res_{|\bar{a}_{s_i}|+|\bar{b}_{s_i}|} \leq_{\xi_i} p^* \res_{|\bar{a}_{s_i}| + |\bar{b}_{s_i}|} \]
and so $q_{s_i} \leq_{\xi_i} p^*$. Finally, for $i = s_k+1,\ldots,s_0$, $p^* \res_{|\bar{a}_i| + |\bar{b}_i|} \ncong \bar{a}_i \bar{b}_i$.
\end{proof}

Let $\bar{a}_s = \ran(p^*)$, and let $\bar{b}_s$ be $\alpha$-free over $\bar{a}_s$ where $\alpha = \max_{t \leq s} H(t)$, and such that $\bar{a}_s \bar{b}_s$ contains the first $s$-many elements of $\scrA$. Let $\bar{c}$ be a new set of constants in $B$ and let $p^{**} = p^* \cup \{\bar{c} \mapsto \bar{b}_s\}$. Let
\[ \ell_{s} = \la p^{**};(\bar{a}_0,\bar{b}_0),(\bar{a}_1,\bar{b}_1),\ldots,(\bar{a}_{s-1},\bar{b}_{s-1}),(\bar{a}_s,\bar{b}_s) \ra.\]
We claim that $\ell_0,\ldots,\ell_s$ is in $P$. That (L1), (L2), and (L3) hold is clear. (L4) and (L5) follow from the choice of $\bar{b}_s$. (P1) is also clear. (P3) follows from the fact that $p^{**} \supseteq q_{s_k}$ and $s_k$ was maximal with $s_k \trianglelefteq s$.

For (P2), if $i \leq s_k$, then since $p^{**} \supseteq q_{s_k}$ and (P2) held at stage $s_k$,
$\ran(p^{**} \res_{|\bar{a}_i| + |\bar{b}_i|}) \cong \bar{a}_i\bar{b}_i$ if and only if $i \trianglelefteq s_k$, and since $s_k \trianglelefteq s$, $i \trianglelefteq s_k$ if and only if $i \trianglelefteq s$ by (B8) and (B9). If $s_k < i < s$, then since $s_k$ is maximal with $s_k \trianglelefteq s$, $i \ntrianglelefteq s$ and by choice of $p^{*}$ in the second claim above, $\ran(p^{**} \res_{|\bar{a}_i| + |\bar{b}_i|}) \ncong \bar{a}_i\bar{b}_i$. The case $i = s$ is clear. Hence $\pi \concat \ell_s \in P$.

Since $p^{**} \supseteq q_{s_k}$, $q_{s_k} \leq_\xi p^{**}$ for all $\xi < \eta$. Given $i < k$, $q_{s_i} \leq_{\xi_i} p^* \subseteq p^{**}$. This completes the proof of the extendability condition.
\end{proof}

By the metatheorem, there is a computable 0-run $\pi = \ell_0 \ell_1 \cdots$ for $(L,P,(\leq_i^L)_{i \leq \eta},E)$. $E(\pi)$ is the diagram of a structure on $\scrB$. For each $j$, let \[ \ell_j = \la p_j; (\bar{a}_0,\bar{b}_0),(\bar{a}_1,\bar{b}_1),\ldots,(\bar{a}_j,\bar{b}_j) \ra.\] Then, along the true stages, by (P3) the $p_i$ are nested, and by (L4) they form a bijection $B \to A$. By definition of $E$, they are an isomorphism $\scrB \to \scrA$.

\begin{lemma}
Let $f: \mc{B} \to \mc{A}$ be an isomorphism. Then $f \geq_T \Delta^0_\eta$.
\end{lemma}
\begin{proof}
Using $f$ we will compute the true path $i_1 \trianglelefteq i_2 \trianglelefteq \ldots$. Then we can compute $\nabla^\eta = \bigcup_{n \in \omega} \nabla^\eta_{i_n}$. We claim that $\ell_j$ is a true stage if and only if
\begin{enumerate}
	\item[($*$)] $\ran(f \res_{|\bar{a}_j| + |\bar{b}_j|}) \cong \bar{a}_j \bar{b}_j$.
\end{enumerate}
Note that ($*$) is computable in $f$, and so this will complete the proof.

If $j$ is a true stage, then $p_j$ extends to an isomorphism $\scrB \to \scrA$. Since $f$ is also an isomorphism, there is an automorphism of $\scrA$ taking $\ran(f \res_{\dom(p_j)})$, as an ordered tuple, to $\ran(p_j)$. By (P2), we have $\ran(p_j \res_{|\bar{a}_j| + |\bar{b}_j|}) \cong \bar{a}_j \bar{b}_j$ and so we have ($*$).

If $j$ satisfies ($*$), then we claim that $j$ is a true stage. Suppose not, and let $p = \bigcup_{n \in \omega} p_{i_n}$ be the isomorphism $\scrB \to \scrA$ along the true path. Let $i_n$ be such that $j < i_n$. Then by (B10), $j \ntrianglelefteq i_n$, and so $\ran(p_{i_n} \res_{|\bar{a}_j| + |\bar{b}_j|}) \ncong \bar{a}_j \bar{b}_j$. Since $p_{i_n} \subseteq p$ and $f$ is also an isomorphism $\scrB \to \scrA$, we have
\[ \ran(f \res_{|\bar{a}_j| + |\bar{b}_j|}) \cong \ran(p_{i_n} \res_{|\bar{a}_j| + |\bar{b}_j|}) \ncong \bar{a}_j \bar{b}_j. \]
This contradicts ($*$). So $j$ is a true stage.
\end{proof}

\begin{lemma}
There is an isomorphism $f: \mc{B} \to \mc{A}$ with $\Delta^0_\eta \geq_T f$.
\end{lemma}
\begin{proof}
Using $\Delta^0_\eta$ we can compute the true path $i_1 \trianglelefteq i_2 \trianglelefteq \cdots$. Then along this path we compute an isomorphism $f = \bigcup_n p_{i_n}$ from $\mc{B} \to \mc{A}$.
\end{proof}

This completes the proof.
\end{proof}

We can improve the statement of the theorem slightly as follows using Knight's theorem on the upwards closure of degree spectra.

\begin{cor}\label{cor:bettermain}
Let $\scrA$ be a computable structure. If $\eta$ is an ordinal and $\scrA$ is not $\Delta^0_\beta$ categorical on any cone for any $\beta < \eta$, then there exists an $\textup{\be}$ such that for all $\textup{\bd} \geq \textup{\be}$, there exists a \textup{\bd}-computable copy $\scrB$ of $\scrA$ such that $\Delta^0_\eta(\textup{\bd})$ computes an isomorphism between $\scrA$ and $\scrB$, and every such isomorphism computes $\Delta^0_\eta(\textup{\bd})$.
\end{cor}

\begin{proof} Take $\be$ as guaranteed by the theorem, and fix $\bd \geq \be$. Let $\scrB$ be as guaranteed by Theorem \ref{thm:main1}. Since $\scrB$ is \bd-computable, by the proof of Knight's upward closure theorem \cite{Knight86}, there exists $\scrC$ such that deg$(\scrC)=\bd$ and such that there exists a $\bd$-computable isomorphism $h: \scrC \iso \scrB$. Now since $\scrA$ is computable and deg$(\scrC)=\bd$, any isomorphism $g: \scrA \iso \scrC$ computes $\bd$. Since $\bd$ computes $h$, $g$ computes the isomorphism $g \circ h :\scrB \iso \scrA$ and hence $\Delta^0_\eta(\bd)$. Moreover, $\bd$ computes an isomorphism between $\scrA$ and $\scrB$, and hence between $\scrA$ and $\scrC$.
\end{proof}

It is now simple to extract Theorem \ref{thm:main1} from the above result.

\begin{proof}[Proof of Theorem \ref{thm:main1}] Let $\scrA$ be a computable structure. By Remark \ref{rem:allstructuresalphacatoncone}, there is an ordinal $\alpha$ such that $\scrA$ is $\Delta^0_\alpha$ categorical on a cone. Let $\alpha \geq 1$ be the least such. By Corollary \ref{cor:bettermain}, there is a cone such that for every $\bf{d}$ in the cone, there exists a $\bf{d}$-computable copy $\scrB$ of $\scrA$ such that every isomorphism between $\scrA$ and $\scrB$ computes $\Delta^0_\alpha(\bd)$. Thus $\scrA$ has strong degree of categoricity $\bf{0^{(\alpha)}}$ on this cone ($\bf{0^{(\alpha-1)}}$ if $\alpha$ is finite).
\end{proof}

%
%

\bibliography{References}

\newcommand{\etalchar}[1]{$^{#1}$}
\begin{thebibliography}{GHK{\etalchar{+}}05}

\bibitem[ACar]{AC}
B.~Anderson and B.~F. Csima.
\newblock Degrees that are not degrees of categoricity.
\newblock {\em Notre Dame J. Form. Log.}, to appear.

\bibitem[AK00]{AshKnight00}
C.~J. Ash and J.~F. Knight.
\newblock {\em Computable structures and the hyperarithmetical hierarchy},
  volume 144 of {\em Studies in Logic and the Foundations of Mathematics}.
\newblock North-Holland Publishing Co., Amsterdam, 2000.

\bibitem[Ash86a]{Ash86a}
C.~J. Ash.
\newblock Recursive labelling systems and stability of recursive structures in
  hyperarithmetical degrees.
\newblock {\em Trans. Amer. Math. Soc.}, 298(2):497--514, 1986.

\bibitem[Ash86b]{Ash86b}
C.~J. Ash.
\newblock Stability of recursive structures in arithmetical degrees.
\newblock {\em Ann. Pure Appl. Logic}, 32(2):113--135, 1986.

\bibitem[Ash87]{AshCatinhypdegrees}
C.~J. Ash.
\newblock Categoricity in hyperarithmetical degrees.
\newblock {\em Ann. Pure Appl. Logic}, 34(1):1--14, 1987.

\bibitem[Ash90]{Ash90}
C.~J. Ash.
\newblock Labelling systems and r.e.\ structures.
\newblock {\em Ann. Pure Appl. Logic}, 47(2):99--119, 1990.

\bibitem[CFS13]{CFS}
B.~F. Csima, J.~N.~Y. Franklin, and R.~A. Shore.
\newblock Degrees of categoricity and the hyperarithmetic hierarchy.
\newblock {\em Notre Dame J. Form. Log.}, 54(2):215--231, 2013.

\bibitem[FFKar]{FFK}
E.~Fokina, A.~Frolov, and I.~Kalimullin.
\newblock Categoricity spectra for rigid structures.
\newblock {\em Notre Dame J. Form. Log.}, to appear.

\bibitem[FKM10]{FKM}
E.~B. Fokina, I.~Kalimullin, and R.~Miller.
\newblock Degrees of categoricity of computable structures.
\newblock {\em Arch. Math. Logic}, 49(1):51--67, 2010.

\bibitem[GD80]{GoncharovDzgoev}
S.~S. Goncharov and V.~D. Dzgoev.
\newblock Autostability of models.
\newblock {\em Algebra and Logic}, 19(1):28--37, 1980.

\bibitem[GHK{\etalchar{+}}05]{GHKMMRS}
S.~S. Goncharov, V.~S. Harizanov, J.~F. Knight, C.~F.~D. McCoy, R.~G. Miller,
  and R.~Solomon.
\newblock Enumerations in computable structure theory.
\newblock {\em Ann. Pure Appl. Logic}, 136(3):219--246, 2005.

\bibitem[HT]{HT}
M.~Harrison-Trainor.
\newblock Degree spectra of relations on a cone.
\newblock preprint.

\bibitem[Kni86]{Knight86}
J.~F. Knight.
\newblock Degrees coded in jumps of orderings.
\newblock {\em J. Symbolic Logic}, 51(4):1034--1042, 1986.

\bibitem[Mar68]{Martin68}
D.~A. Martin.
\newblock The axiom of determinateness and reduction principles in the
  analytical hierarchy.
\newblock {\em Bull. Amer. Math. Soc.}, 74:687--689, 1968.

\bibitem[Mar75]{Martin75}
D.~A. Martin.
\newblock Borel determinacy.
\newblock {\em Ann. of Math. (2)}, 102(2):363--371, 1975.

\bibitem[McC02]{McCoy02}
C.~F.~D. McCoy.
\newblock Finite computable dimension does not relativize.
\newblock {\em Arch. Math. Logic}, 41(4):309--320, 2002.

\bibitem[Mil09]{Miller2009}
R.~Miller.
\newblock {$d$}-computable categoricity for algebraic fields.
\newblock {\em J. Symbolic Logic}, 74(4):1325--1351, 2009.

\bibitem[Mona]{Montalban}
A.~Montalb\'an.
\newblock Priority arguments via true stages.
\newblock To appear in the Journal of Symbolic Logic.

\bibitem[Monb]{MontalbanSR}
A.~Montalb\'an.
\newblock A robuster scott rank.
\newblock preprint.

\bibitem[Rem81]{Remmel}
J.~B. Remmel.
\newblock Recursively categorical linear orderings.
\newblock {\em Proc. Amer. Math. Soc.}, 83(2):387--391, 1981.

\end{thebibliography}
\bibliographystyle{alpha}

\end{document}